\documentclass{amsart}

\newcommand{\R}{\mathbb{R}}

\def\eps{\varepsilon}

\usepackage{amsthm,amsmath,amsfonts, amssymb}
\usepackage[dvips]{graphicx}

\usepackage{psfrag}

\newtheorem{theorem}[equation]{Theorem}
\newtheorem{proposition}[equation]{Proposition}
\newtheorem{lemma}[equation]{Lemma}
\newtheorem{cor}[equation]{Corollary}

\theoremstyle{remark}
\newtheorem{remark}[equation]{Remark}
\theoremstyle{definition}
\newtheorem{definition}[equation]{Definition}

\numberwithin{equation}{section}

\begin{document}

\author{Bruno Colbois}
\address{Institut de Math\'ematiques de Neuch\^atel, rue Emile-Argand 11,
  2009 Neuch\^atel, Suisse.}
\email{bruno.colbois@unine.ch}

\author{Alexandre Girouard}
\address{Institut de Math\'ematiques de Neuch\^atel, rue Emile-Argand 11,
  2009 Neuch\^atel, Suisse.}
\email{alexandre.girouard@unine.ch}

\author{Mette Iversen}
\address{School of Mathematics
University Walk, Clifton, Bristol BS8 1TW, U.K. }
\email{mette.iversen@bristol.ac.uk}

\title[Uniform stability of the Dirichlet spectrum]{Uniform stability of the Dirichlet spectrum for rough
    outer perturbations}

\begin{abstract}
  The goal of this paper is to study the Dirichlet eigenvalues of
  bounded domains $\Omega\subset\Omega'$. With a local spectral
  stability requirement on $\Omega$,
  we show that the difference of the Dirichlet
  eigenvalues of $\Omega'$ and $\Omega$ is explicitly
  controlled from above in terms of the first eigenvalue of
  $\Omega'\setminus\overline{\Omega}$ and of geometric
    constants depending on the inner domain $\Omega$. In particular,
     $\Omega'$ can be an arbitrary bounded domain.
\end{abstract}

\maketitle

\section{Introduction and results}

Let $M$ be a complete smooth Riemannian manifold. The Dirichlet eigenvalues
of a bounded domain\footnote{A domain is a connected open set.}
$\Omega\subset M$ are denoted
$$\lambda_1(\Omega)<\lambda_2(\Omega)\leq\lambda_3(\Omega)\leq\cdots\nearrow\infty.$$
In the Euclidean case, there is a vast literature on
\emph{spectral stability} of the Dirichlet spectrum under
perturbation of the domain. The aim is to show that if $\Omega'$
is another domain which is, in some sense, geometrically close to
$\Omega$, then its Dirichlet eigenvalues are close to those of
$\Omega$.  See for instance the papers~\cite{DanerJDE,HenrotCiber}
and the references therein, where spectral stability is studied in
particular from the point of view of $\gamma$-convergence.
Explicit control of the spectrum has been studied for example
in~\cite{BLL,lemenant1}.

%  The case where the boundary $\partial \Omega$ satisfies some extra-conditions (like an Hardy inequality) and
%  $\Omega'$ is an $\epsilon$-neighborhood of $\Omega$ (that has to be defined precisely):
%  we can again get explicit estimates of $\lambda_k(\Omega) - \lambda_k(\Omega')$ in terms of  power of
%  $\epsilon$, see for example E.B. Davies []. In the sequel, we will call this \emph{local estimates} and say that
% the spectrum of $\Omega$ is \emph{locally stable} when we have such estimates.
% As we will make use of this approach in this paper, we will come back on
%  it with much more details.

In this paper, we are interested in obtaining explicit estimates
in the situation where the two domains $\Omega$ and $\Omega'$
might not be geometrically close. For domains $\Omega\subset
\Omega',$ the difference of eigenvalues
$|\lambda_k(\Omega)-\lambda_k(\Omega')|$ will be controlled in
terms of the fundamental tone
$\lambda_1(\Omega'\setminus\overline{\Omega})$. In particular, the
domains $\Omega'$ and $\Omega$ can have very different shapes, and
the volume of $\Omega'$ can be large compared to that of $\Omega$.
Some natural conditions on the inner domain $\Omega$ need to be
assumed, but $\Omega'$ can be any bounded open set. We will do this by
combining local estimates based on the work of E.B.
Davies~\cite{Davies3}, with global estimates based on the work of
the first author and J. Bertrand~\cite{BertrandColbois}. Note that
in this paper constants will depend only on the stated parameters.

\subsection{Statement of results}
Our goal is to estimate
$|\lambda_k(\Omega)-\lambda_k(\Omega')|$ in the situation where
$\Omega\subset\Omega'$ are bounded domains. Throughout we will use the notation
\begin{gather*}
  \Omega^{\eps}=\{x \in M: d(x,\Omega)<\eps\},\quad\quad
  \Omega_{\eps}=\{x \in \Omega : d(x,\Omega^c)>\eps\},
\end{gather*}
and
\begin{gather*}
  \mu=\lambda_1(\Omega'\setminus\overline{\Omega}),\quad\lambda=\lambda_k(\Omega'),
\end{gather*}
with the convention $\lambda_1(\emptyset)=\infty$.
Observe that because
$\Omega\subset\Omega^\epsilon\cap\Omega'\subset\Omega'$, it follows from monotonicity
of the Dirichlet eigenvalues that
\begin{gather}\label{Inequality:Triangle}
  0\leq\lambda_k(\Omega)-\lambda_k(\Omega')=\bigl(\lambda_k(\Omega^\epsilon\cap\Omega')-\lambda_k(\Omega')\bigr)
  +\bigl(\lambda_k(\Omega)-\lambda_k(\Omega^\epsilon\cap\Omega')\bigr),
\end{gather}
where the two terms on the right hand side are non-negative.
% It follows that
% \begin{gather}
%   |\lambda_k(\Omega)-\lambda_k(\Omega')|=
%   |\lambda_k(\Omega^\epsilon\cap\Omega')-\lambda_k(\Omega')|+|\lambda_k(\Omega)-\lambda_k(\Omega^{\epsilon})|.
% \end{gather}
We estimate these two terms separately, and call them the
\emph{global} and the \emph{local} term respectively.

\subsubsection*{Global estimates}
In Section~\ref{SectionProofMain}, we prove the following theorem
allowing control of the first term in the right hand side of
(\ref{Inequality:Triangle}) when
$\mu=\lambda_1(\Omega'\setminus\overline{\Omega})$ is large.

\begin{theorem}\label{ThmMain}
  There are constants $a_k, b_k\geq 1$  defined by the recurrence
  relations~\eqref{defak} and~\eqref{defbk}, with the
  following property:
  For each $\alpha\in(0,\frac{1}{4})$,
  if
  \begin{gather}\label{ConditionThmMain}
    32\left(\frac{\lambda}{\mu}\right)^{\frac{1}{2}-2\alpha}a_k\leq 1
  \end{gather}
  then
  \begin{gather}\label{eqThmMainintro}
    |\lambda_k(\Omega^{\epsilon}\cap\Omega')-\lambda_k(\Omega')|\leq
    b_k
    \left(
    \frac{\lambda}{\mu}
    \right)^{\frac{1}{2}-2\alpha}\lambda,
  \end{gather}
  for the choice
  $$\epsilon=2
  \left(
    \frac{\lambda}{\mu}
  \right)^\alpha
  \frac{1}{\sqrt{\lambda}}.$$
\end{theorem}
\begin{remark}
\begin{itemize}
\item[]
\item Monotonicity implies that the eigenvalue $\lambda=\lambda_k(\Omega')$ is bounded above by
  $\lambda_k(\Omega)$, so that the numerator of the right hand
  side of \eqref{eqThmMainintro}
  is bounded above in terms of $\Omega$ and $k$ only.
\item Inequality~\eqref{eqThmMainintro} is invariant under
  homothetic scaling of the domains. Of course, the value of
  $\epsilon$ must also be modified in accordance with its definition.
\end{itemize}
\end{remark}
The strategy consists in a geometrical
approach inspired by a special case of the proof of Theorem 3.3
in~\cite{BertrandColbois}.

\subsubsection*{Local estimates}
In Section~\ref{SectionUnifLocSpectralStab}, we describe classes
of domains $\Omega\subset\mathbb{R}^n$ for which we have uniform control of the
second term in the right hand side of (\ref{Inequality:Triangle})
for each $k\in\mathbb{N}$. They consist of domains
$\Omega$ with thickenings $\Omega^\epsilon$ satisfying a weak
Hardy inequality.
\begin{definition}
  A domain $\Omega$ satisfies a \emph{weak Hardy inequality} with constants
  $a, b$ if for each $u\in C^{\infty}_0(\Omega)$,
  $$\int_{\Omega} \frac{u^2}{\delta^2} \le a
  \int_{\Omega}\left(|\nabla u|^2 + b u^2 \right),$$
  where $\delta$ denotes the distance function to the boundary of $\Omega$.
\end{definition}

The following illustrates our use of the Hardy inequality for a
particular class of domains. (See Section \ref{SubsectionFamilies} for the definition of the uniform external rolling ball condition.)
\begin{lemma}\label{LocalStabLemmaIntro}
Let $\eps_0>0, r_0>0$.
  Let $\mathcal A= \mathcal A(\eps_0,r_0,n)$ be the family of open, bounded sets $\Omega$
  in  $\R^n, \; n\ge 2,$ with inradius bounded below by
  $r_0$ satisfying a uniform external rolling ball condition with parameter $\eps_0$.
 Then there exist positive constants $\gamma=\gamma(n)$,  $C_k=C_k(n,\eps_0,r_0)$  and $\eps_k=\eps_k(n,\eps_0,r_0)$ such that
  for any $\Omega \in \mathcal A$ and any $\eps \le \min(\eps_0/2,\eps_k),$
  \begin{equation}
    0 < \lambda_k(\Omega)-\lambda_k(\Omega^\eps) \le C_k \eps^{\gamma}.
  \end{equation}
\end{lemma}
\begin{remark}
    In contrast to the global estimate of Theorem~\ref{ThmMain}, this local estimate is not scaling invariant.
    Indeed, the proof of Lemma~\ref{LocalStabLemmaIntro} is based on the fact
    that the geometric hypotheses imply uniform control of the
    constants $a$ and $b$ appearing in the Hardy inequalities for the
    thickenings $\Omega^\epsilon$. This control allows the application of a
    result of E.B. Davies from~\cite{Davies3}, which is inherently non-homogeneous.
    However, for convex domains, we were able to obtain invariant
    local bounds (see Proposition~\ref{PropositionLocalStabConvex}).
\end{remark}

\begin{cor}\label{CoroIntro}
 Under the hypotheses of Lemma~\ref{LocalStabLemmaIntro},  there exist constants
  $a_k,b_k,\gamma=\gamma(n), C_k=C_k(n,\eps_0,r_0)$  and $\eps_k=\eps_k(n,\eps_0,r_0)$ such that for each domain
  $\Omega\in\mathcal{A}$ the following holds.
 Let $\alpha\in(0,\frac{1}{4})$, and suppose
 that~\eqref{ConditionThmMain} and
 $$\left(
 \frac{\lambda}{\mu}
 \right)^\alpha
 \frac{1}{\sqrt{\lambda}}\leq \min(\eps_0,\eps_k).
 $$
 holds.
 Then
 \begin{gather}\label{IneqCorIntro}
   |\lambda_k(\Omega)-\lambda_k(\Omega')|\leq
   b_k\left(\frac{\lambda}{\mu}\right)^{\frac{1}{2}-2\alpha}\lambda
   +
   \frac{C_k}{\lambda^{\frac{\gamma}{2}}}\left(
   \frac{\lambda}{\mu}
   \right)^{\gamma\alpha}.
 \end{gather}
 Taking for example $\alpha=\frac{1}{2(2+\gamma)}$ gives
  \begin{gather}\label{Ineq:locglobsixth}
    |\lambda_k(\Omega)-\lambda_k(\Omega')|\leq
    \left(b_k \lambda^{\frac{4+3\gamma}{2(2+\gamma)}}+C_k \lambda^{-\frac{\gamma(1+\gamma)}{2(2+\gamma)}}\right) \mu^{-\frac{1}{2(2+\gamma)}}.
  \end{gather}
\end{cor}

In Section~\ref{SectionUnifLocSpectralStab}, similar results will
be proved for families defined in terms of a cone condition and a capacity density condition. In
Section~\ref{SectionExamples}, we give two simple examples to
illustrate the necessity of two of the geometric conditions
imposed in Section~\ref{SubsectionFamilies}. Our goal for this
section is to give some simple criteria implying  local stability.

\subsubsection*{Proximity of eigenspaces}
In Section~\ref{SectionProximityEigenspace}, we control the
proximity of the eigenspaces on
$\Omega'$ and $\Omega$ in terms of
$\mu=\lambda_1(\Omega'\setminus\overline{\Omega})$.
Stability of eigenfunctions for the Dirichlet problem is well known, so
our contribution is to provide an explicit control in terms of $\mu$.

\subsection{Discussion of results}
Stability of the Dirichlet spectrum is closely related to the
stability of the corresponding Dirichlet problem
\begin{gather*}
  -\Delta u=f\mbox{ in }\Omega,\quad
  u=0\mbox{ on }\partial\Omega.
\end{gather*}
Indeed, it is well known that estimates for the associated resolvent
operator $R_\Omega$ translate into corresponding bounds for the
eigenvalues. This has been studied from the point of view of various
interrelated notions of convergence of domains. See~\cite[Chapter 2.3]{HenrotVert} for
an enlightening discussion. In~\cite{DanerJDE} it is
proved, under rather weak assumptions on $\Omega$, that if a
sequence $\Omega_n$ of domains containing $\Omega$ is such that
$$\lim_{n\rightarrow\infty}\lambda_1(\Omega_n\setminus\overline{\Omega})=\infty,$$
then  $\Omega_n$ $\gamma$-converges to $\Omega$, which implies
$$\lim_{n\rightarrow\infty}\lambda_k(\Omega_n)=\lambda_k(\Omega).$$
In this situation, our results provide explicit control of the
difference $|\lambda_k(\Omega)-\lambda_k(\Omega_n)|$ in terms of
$\mu$, and so control of the rate of convergence.

As mentioned, explicit estimates of the difference
$|\lambda_k(\Omega)-\lambda_k(\Omega')|$ in terms of for example
the measure of the symmetric difference $\Omega'\Delta\Omega$
in~\cite{BLL} have also been given previously. Our estimate in
terms of $\mu=\lambda_1(\Omega'\setminus\overline{\Omega})$ allows
control in addition when the measure of $\Omega'\setminus\Omega$ is
large. The results of ~\cite{BLL}
 are valid for classes of Lipschitz domains described in~\cite[Section
 2.3]{BLL}. Our results complement this by providing a selection
 of geometric conditions under which we have control of the
 spectrum.

\subsection*{Acknowledgements}
We would like to thank Yuri Netrusov, Dorin Bucur, Michiel van den
Berg, Iosif Polterovich and Antoine Lemenant for valuable
discussions. Part of this
work was accomplished during the conference \emph{Shape optimization
  problems and spectral theory} which was held at CIRM (Marseille)
from May 28 to June 2, 2012. We would like to thank the organizers as
well as CIRM for this great opportunity.

\section{Global estimates}
\label{SectionProofMain} The goal of this section is to prove
Theorem~\ref{ThmMain}. Let $\Omega\subset\Omega'$ be bounded
domains in the complete smooth Riemannian manifold $M$. Let
$(f_i)_{i\in\mathbb{N}}$ be an orthonormal basis of $L^2(\Omega')$
corresponding to the Dirichlet eigenvalues $\lambda_i(\Omega')$.
Fix $\epsilon>0$ and let $\eta:M\rightarrow\mathbb{R}$ be a cutoff
function such that
\begin{gather*}
  0\leq\eta\leq
  1,\quad\quad|\nabla\eta|\leq\frac{1}{\epsilon},\\
  \eta\equiv
  \begin{cases}
    1\mbox{ in }\Omega^\epsilon,\\
    0\mbox{ in }M\setminus\Omega^{2\epsilon}.
  \end{cases}
\end{gather*}
For each $k\in\mathbb{N}$, the function
$$\psi_k:=\eta f_k\in H^1_0(\Omega^{2\epsilon}\cap\Omega')$$
will be used as test function in the variational characterization
of $\lambda_k(\Omega^{2\epsilon})$ thanks to a result of the first
author and J. Bertrand~\cite[Lemma 3.13]{BertrandColbois}. This
result is stated here in a slightly modified version for
convenience.
\begin{lemma}\label{LemmaBertrandColbois}
  Let $a_1=1$ and for $k>1$, recursively define
  \begin{gather}\label{defak}
    a_k=1+\sum_{i=1}^{k-1}a_i^2.
  \end{gather}
  Let $\rho$ be a positive number such that $4\rho a_k\leq 1$.
  Let $b_1=4$ and for $k>1$, recursively define
  \begin{gather}\label{defbk}
    b_k=(1+8a_k)((1+\rho b_{k-1})(1+8a_k)+1).
  \end{gather}
  For each $\lambda>0$, the following holds:
  Let $q$ be a quadratic form on an Euclidean space $E$ of dimension
  $k$. Let $\psi_1,\cdots,\psi_k\in E$ be such that for each $i,j\leq k$,
  \begin{gather}\label{eqAlmostOrtho}
    |\langle \psi_i,\psi_j\rangle-\delta_{ij}|\leq \rho
    \mbox{ and }
    q(\psi_i)\leq\lambda(1+\rho).
  \end{gather}
  Then there exists an orthonormal basis $(F_i)_{1\leq i\leq k}$ of
  $E$ such that for each $i\in\{1,2,\cdots,k\}$
  \begin{gather*}
    q(F_i)\leq\lambda(1+\rho b_k).
  \end{gather*}
\end{lemma}
The proof of Lemma~\ref{LemmaBertrandColbois} differ only slightly of the original proof, and
the modifications will be presented in
Section~\ref{SecionProofBertrandColbois}.
In our situation, the quadratic form is the Dirichlet energy
defined on the space $E=\mbox{span}(\psi_1,\cdots,\psi_k)$ in
$L^2(\Omega^{2\epsilon})$.
\begin{lemma}\label{LemmaRayleighCutoff}
  For each $j\in\{1,2,\cdots,k\}$, the Dirichlet energy of the test
  function $\psi_j=\eta f_j$ satisfies
  \begin{align}\label{InequalityLemmaRayleigh}
    \|\nabla\psi_j\|^2&\leq
    \lambda+\Lambda+
      2\Lambda^{1/2}
      \lambda^{1/2},
  \end{align}
where $\Lambda=2\left(\frac{1+\epsilon^2\lambda}{\epsilon^4\mu}\right)$.
\end{lemma}
\begin{proof}[Proof of Lemma~\ref{LemmaRayleighCutoff}]
  Writing $f=f_j$ and $\psi=\eta f$ to simplify notations,
  direct computation using the definition of $\psi$ and of the cutoff function
  $\eta$ leads to
  \begin{align}
    \int_{\Omega^{2\epsilon}}|\nabla\psi|^2\nonumber
    &=\int_{\Omega^{2\epsilon}\setminus\Omega^\epsilon}|\nabla\eta|^2f^2+
    2\int_{\Omega^{2\epsilon}\setminus\Omega^\epsilon}\eta f\nabla\eta\cdot\nabla
    f+
    \int_{\Omega^{2\epsilon}}\eta^2|\nabla f|^2\nonumber\\
    &\leq
    \frac{1}{\epsilon^2}\underbrace{\int_{\Omega^{2\epsilon}\setminus\Omega^\epsilon}f^2}_A+
    \frac{2}{\epsilon}\int_{\Omega^{2\epsilon}\setminus\Omega^\epsilon}|f||\nabla
    f|+\int_{\Omega^{2\epsilon}}|\nabla f|^2\nonumber\\
    &\leq
    \frac{1}{\epsilon^2}A+
    \frac{2}{\epsilon}A^{1/2}
    \left(
      \int_{\Omega^{2\epsilon}\setminus\Omega^\epsilon}|\nabla  f|^2
    \right)^{1/2}
    +\int_{\Omega^{2\epsilon}}|\nabla f|^2\nonumber\\
    &\leq\label{var1}
    \frac{1}{\epsilon^2}A+
    \frac{2}{\epsilon}A^{1/2}
    \lambda^{1/2}+\lambda.
  \end{align}
  In order to give an upper bound for
  \begin{gather*}
    A=\int_{\Omega^{2\epsilon}\setminus\Omega^\epsilon}f^2,
  \end{gather*}
  an auxiliary cutoff function
  $\chi:M\rightarrow\mathbb{R}$ is introduced, satisfying
  \begin{gather}\label{deffunctionchi}
    0\leq\chi\leq
    1,\quad |\nabla\chi|\leq\frac{1}{\epsilon},\nonumber\\
    \chi\equiv
    \begin{cases}
      1\mbox{ in }M\setminus\Omega^{\epsilon},\\
      0\mbox{ in }\Omega.
    \end{cases}
  \end{gather}
  It follows that
  $$A=\int_{\Omega^{2\epsilon}\setminus\Omega^\epsilon}\chi^2f^2\leq
  \int_{M\setminus\Omega}\chi^2f^2=\|\chi f\|^2.$$
  The function $\chi f$ is then used in the variational characterization of
  $\mu$~:
  \begin{align*}
    \mu\|\chi f\|^2
    &\leq \int_{\Omega'\setminus\overline{\Omega}}|\nabla (\chi f)|^2
    \leq 2\int_{\Omega'\setminus\overline{\Omega}} (|\nabla\chi|^2f^2+\chi^2 |\nabla
    f|^2)\\
    &\leq
    \frac{2}{\epsilon^2}\int_{\Omega'\setminus\overline{\Omega}}f^2+2\int_{\Omega'\setminus\overline{\Omega}}|\nabla f|^2\\
    &\leq
    2\left(\frac{1}{\epsilon^2}+\lambda\right).
  \end{align*}
  It follows that
  \begin{gather}\label{BoundOnA}
    A\leq\|\chi f\|^2\leq 2\left(\frac{1+\epsilon^2\lambda}{\epsilon^2\mu}\right),
  \end{gather}
  which is substituted back into inequality~(\ref{var1}) to complete
  the proof.
\end{proof}
The following lemma shows that the test functions $\psi_i$ form an
almost orthonormal family in $L^2$.
\begin{lemma}\label{LemmaAlmostOrtho}
  For each $i,j\leq k$,
  \begin{align*}
    |\langle\psi_i,\psi_j\rangle-\delta_{ij}| \leq 8\epsilon^2\Lambda,
  \end{align*}
  where
  $\Lambda=2\left(\frac{1+\epsilon^2\lambda}{\epsilon^4\mu}\right)$.
\end{lemma}

\begin{proof}[Proof of Lemma~\ref{LemmaAlmostOrtho}]
  For the case $i=j$,
  the inequality~(\ref{BoundOnA}) implies
  \begin{align*}
    \|\psi_i\|^2\geq \int_{\Omega^\epsilon} f_i^2&=1-\int_{M\setminus\Omega^\epsilon}\chi^2f_i^2
    \geq
    1-2\left(\frac{1+\epsilon^2\lambda}{\epsilon^2\mu}\right),
  \end{align*}
  where the cutoff function $\chi$ has been defined in (\ref{deffunctionchi}).
  Since $\|\psi_i\|^2\leq \|f_i\|^2=1$, this implies
  \begin{gather*}
    |\langle\psi_i,\psi_i\rangle-1| \leq  2\left(\frac{1+\epsilon^2\lambda}{\epsilon^2\mu}\right)=\epsilon^2\Lambda.
  \end{gather*}
  For $i\neq j$,
  \begin{align*}
    |\langle\psi_i,\psi_j\rangle|=
    \vert \int_{\Omega^{2 \epsilon}} \eta^2f_if_j\vert =
    \vert\int_{\Omega^{2 \epsilon}}(\eta^2-1)f_if_j+ \int_{\Omega^{2 \epsilon}}f_if_j\vert.
  \end{align*}
  As $f_i$ and $f_j$ are orthogonal on $\Omega'$,
  \begin{align*}
    \vert \int_{\Omega^{2 \epsilon}}f_if_j\vert =  \vert \int_{\Omega'\setminus\Omega^{2 \epsilon}}f_if_j\vert
    \le  2 \int_{\Omega'\setminus\Omega^{2 \epsilon}}(f_i^2+f_j^2)\le 2(\Vert \chi f_i\Vert^2+\Vert \chi f_j
    \Vert^2),
  \end{align*}
  which together with $1-\eta^2 \le \chi^2$ implies
  \begin{align*}
    \vert\int_{\Omega^{2 \epsilon}}(\eta^2-1)f_if_j\vert  \le 2 (\Vert \chi f_i\Vert^2+\Vert \chi f_j \Vert^2).
  \end{align*}
  Combining this with inequality~(\ref{BoundOnA}) and noting that
  $\max(\lambda_i,\lambda_j)\leq\lambda_k(\Omega')$, then  gives
  \begin{align*}
    \vert \int_{\Omega^{2 \epsilon}} \psi_i\psi_j\vert \le 16
    \frac{1+\epsilon^2\lambda}{\epsilon^2\mu}.
  \end{align*}
\end{proof}

\begin{proof}[Proof of Theorem~\ref{ThmMain}]
  Given $\alpha\in(0,\frac{1}{4})$, let
  $$\epsilon=\left(\frac{\lambda}{\mu}\right)^\alpha\frac{1}{\sqrt{\lambda}}.$$
  Note that by condition~\eqref{ConditionThmMain},
  $\epsilon<1/\sqrt{\lambda}$, which
  %$0<\frac{1}{2}-2\alpha< 1-4\alpha< 1-2\alpha,$
  %which with $\lambda\leq\mu$
  implies $\Lambda\leq
  \frac{4}{\epsilon^4\mu}$.
  Thus
  \begin{gather*}
    \Lambda+2\sqrt{\Lambda}\sqrt{\lambda}
    \leq\frac{4}{\epsilon^4\mu}+\frac{4\sqrt{\lambda}}{\epsilon^2\sqrt{\mu}}=
    4\left(\frac{\lambda}{\mu}\right)^{1-4\alpha}\lambda+
    4\left(\frac{\lambda}{\mu}\right)^{\frac{1}{2}-2\alpha}\lambda
    \leq
    8\left(\frac{\lambda}{\mu}\right)^{\frac{1}{2}-2\alpha}\lambda.
  \end{gather*}
  Moreover
  \begin{gather*}
    \epsilon^2\Lambda\leq\frac{4}{\epsilon^2\mu}=4\left(\frac{\lambda}{\mu}\right)^{1-2\alpha}
    \leq 8\left(\frac{\lambda}{\mu}\right)^{\frac{1}{2}-2\alpha}.
  \end{gather*}
  The hypotheses of Lemma~\ref{LemmaBertrandColbois} then holds with
  $$\rho=8\left(\frac{\lambda}{\mu}\right)^{\frac{1}{2}-2\alpha}.$$

\end{proof}

\section{Uniform local stability}
\label{SectionUnifLocSpectralStab}
A family $\mathcal A$ of domains $\Omega \subset M$ is \emph{uniformly
locally stable} if there exists $\epsilon_0>0$ such that
for each $\Omega \in \mathcal A$, a uniform upper bound for
$|\lambda_k(\Omega) -\lambda_k(\Omega^{\eps})|$ holds, in terms of
$\epsilon\leq\epsilon_0$, $k\in\mathbb{N}$, and some geometric quantities depending only on
$\Omega$. The family $\mathcal{A}$
described in the introduction (see Lemma~\ref{LocalStabLemmaIntro}) is a prototypical example of uniform
local stability. Other examples of such
families will be given in Section~\ref{SubsectionFamilies}.

\subsection{Motivating examples}\label{SectionExamples}

The goal of the present section is to give two simple
examples to illustrate the necessity of the
geometric conditions imposed in the construction of the
families in Section~\ref{SubsectionFamilies}.
The first illustrates the need for a lower bound
on the inradius, while the second for a condition of the form
$(\Omega^\eps)_{N\eps} \subset \Omega$ for all $\eps \le
  \eps_0$, as required in Lemma \ref{LocalStabLemmaIntro}.

\subsubsection*{Example $1$}
  The Dirichlet eigenvalues of the ball $\Omega=B(0,r)\subset\mathbb{R}^n$ are
  \begin{gather*}
    \lambda_k(\Omega)=\frac{c(n,k)}{r^2},\quad k\in\mathbb{N},
  \end{gather*}
  where $c(n,k)$ is the $k$-th eigenvalue of a ball of radius one in
  $\mathbb{R}^n$.
  It follows that
  \begin{align*}
    \lambda_k(\Omega)-\lambda_k(\Omega^\eps)
    =c(n,k)\eps
    \left(
      \frac{2r+\epsilon}{r^2(\epsilon+r)^2}
    \right).
  \end{align*}
  In order to have a uniform upper bound on
  $\lambda_k(\Omega)-\lambda_k(\Omega^\eps)$, it is necessary to
  consider balls $B(0,r)$ of radius $r$ bounded below, say
  by $r_0>0$.
  Consider the family
  \begin{gather*}
    \mathcal{A}=\mathcal{A}(n,r_0)=\bigl\{B(p,r)\,:\,p\in\mathbb{R}^n,r\geq r_0\bigr\}.
  \end{gather*}
  For any ball $\Omega\in\mathcal{A}$, and for any
  $\epsilon\leq\epsilon_0:=r_0$, one easily sees that
  \begin{gather*}
    \lambda_k(\Omega)-\lambda_k(\Omega^\eps)\leq \frac{3c(n,k)}{r_0^2}\frac{\epsilon}{r_0}.
  \end{gather*}
  In the more general context of Section~\ref{SubsectionFamilies},
  this will translate into lower bounds on the inradius.

\subsubsection*{Example $2$}
  Consider the family
  $\mathcal{A}=\bigl\{\Omega_t\,:\,t\in (0,1)\bigr\}$
  \begin{figure}[h]
    \centering
    \psfrag{1}[][][1]{$t$}
   \psfrag{2}[][][1]{$1$}
   \includegraphics[width=4cm]{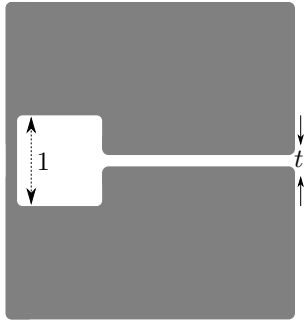}
    \caption{The domains $\Omega_t$, $0<t<1$.}
    \label{figure:one}
  \end{figure}
  % I have edited the eps file by hand to replace <...>Tj by (1)Tj, which
  % is what is referenced in the above psfrag. This is used because
  % inkscape does not export text in .eps files as a string but rather as
  % something encoded.
  described in  Figure~\ref{figure:one}. Some of the features of this family are:
  \begin{enumerate}
  \item The inradius of $\Omega_t$ is uniformly bounded.
  \item The boundary $\partial\Omega_t$ is smooth and its curvature is
    uniformly bounded.
  \end{enumerate}
  In spite of these two properties, this family is not uniformly
  locally stable. Indeed, for each $\epsilon_0>0$, choosing
  $t\leq\epsilon_0$
  leads to discontinuous variations of  $\lambda_k$ as $\epsilon$
  varies from $0$ to $\epsilon_0$. This follows from the fact that
  $\Omega_t^\epsilon$ is completely different from $\Omega_t$ since at
  $\epsilon=\epsilon_0/2$, the domain $\Omega_t^\epsilon$ becomes
  doubly connected.

  In Lemma~\ref{LocalStabLemmaIntro}, this situation was avoided by
  requiring $(\Omega^\eps)_{N\eps} \subset \Omega$ for all $\eps \le
  \eps_0$, and we will need a condition of this type to be satisfied throughout.

\subsection{Convex domains}
Local spectral stability is particularly simple for convex
domains. Given $r_0>0$, consider the family
\begin{gather*}
  \mathcal{A}=\mathcal{A}(n,r_0)=
  \bigl\{
  \Omega\subset\mathbb{R}^n\,:\Omega\mbox{ is convex, }
  \mbox{inradius}(\Omega)\geq r_0
  \bigr\}.
\end{gather*}
\begin{proposition}\label{PropositionLocalStabConvex}
  For any domain $\Omega\in\mathcal{A}$,
  \begin{gather*}
    \lambda_k(\Omega)-\lambda_k(\Omega^\epsilon)\leq
    \frac{c(n,k)}{r_0^2}\left(\frac{2\epsilon}{r_0}+\frac{\epsilon^2}{r_0^2}\right).
  \end{gather*}
\end{proposition}
\begin{lemma}\label{LemmaDilationContraction}
  Let $\Omega$ be a convex domain. Then $(\Omega^\epsilon)_\epsilon=\Omega$.
\end{lemma}
\begin{proof}[Proof of Lemma~\ref{LemmaDilationContraction}]
  A point $x\in\mathbb{R}^n$ lies in $(\Omega^\epsilon)_\epsilon$ if
  and only if
  $\overline{B(x,\epsilon)}\subset\Omega^\epsilon$. Suppose that
  $x\notin\Omega$. Then there exists a hyperplane separating $x$ from
  $\Omega$. This implies the existence of
  $y\in\overline{B(x,\epsilon)}$ such that
  $d(y,\Omega)\geq\epsilon$, which contradicts
  $\overline{B(x,\epsilon)}\subset\Omega^\epsilon$.
\end{proof}
\begin{lemma}\label{LemmaConvexDilation}
  Let $r_0>0$. Let $\Omega\subset\mathbb{R}^n$ be a bounded convex
  Euclidean domain such that $B(x,r_0)\subset\Omega$. Then
  $$\Omega^\epsilon\subset H(\Omega)$$
  where $H$ is an homothety of factor $1+\epsilon/r_0$ with
  center $x$.
\end{lemma}
  \begin{figure}[h]
    \centering
    \psfrag{x}[][][1]{$x$}
    \psfrag{x'}[][][1]{$x'$}
    \psfrag{y}[][][1]{$y$}
    \psfrag{y'}[][][1]{$y'$}
    \psfrag{z}[][][1]{$z$}
    \includegraphics[width=7cm]{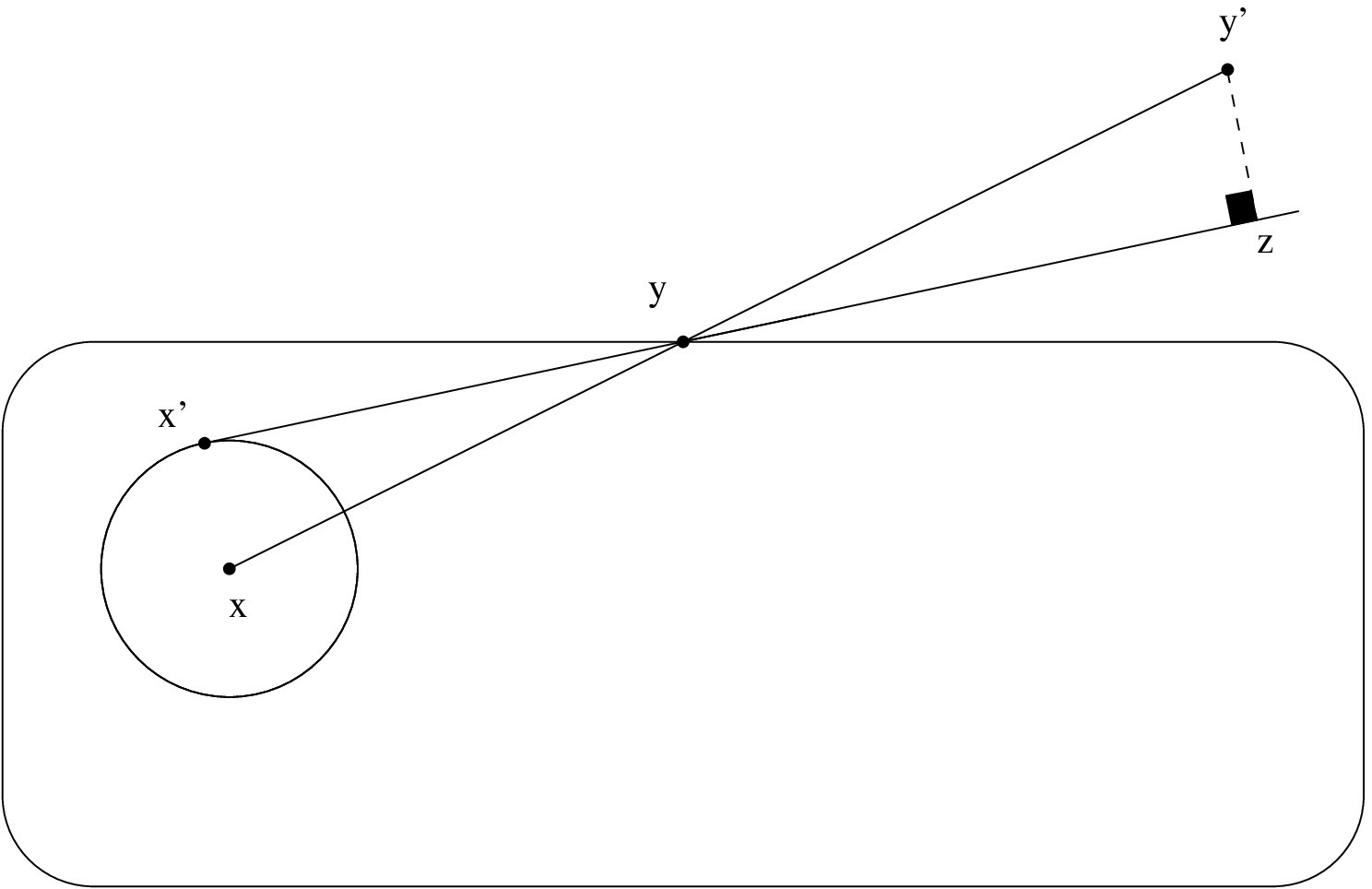}
    \caption{Proof of Lemma~\ref{LemmaConvexDilation}}
    \label{figure:convex}
  \end{figure}
\begin{proof}
  Fix $y\in\partial\Omega$.
  There is a unique point  $y'\in\partial\Omega^\epsilon$ such that
  $y\in xy'$, where $xy'$ denotes the segment connecting $x$ and $y'$.
  There also exists a unique point $x'\in\partial B(x,r_0)$
  such that the segment $x'y$ is tangent to $\partial B(x,r_0)$. Let $z$ be the orthogonal
  projection of $y'$ onto the line through $x'$ and $y$.
  It follows from Thales' theorem that we can compare the lengths of the segments to get
  \begin{gather*}
    \frac{\overline{zy'}}{\overline{yy'}}=
    \frac{\overline{xx'}}{\overline{xy}}=\frac{r_0}{\overline{xy}}.
  \end{gather*}
  Together with Lemma \ref{LemmaDilationContraction} this implies
  \begin{gather*}
    \frac{\overline{yy'}}{\overline{xy}}=
    \frac{\overline{zy'}}{r_0}<\frac{\epsilon}{r_0}.
  \end{gather*}
  Noting that $\overline{yy'}=\overline{xy'}-\overline{xy}$, we then have
  \begin{gather*}
    \frac{\overline{xy'}}{\overline{xy}}<1+\frac{\epsilon}{r_0}.
  \end{gather*}
  This gives $y'\in H(\Omega),$
  and so completes the proof as $y\in\partial\Omega$ was arbitrary.
\end{proof}

\subsection{A result of E. B. Davies}
%change alpha to sigma
In studying the behaviour of the spectrum of the outer perturbation
$\Omega^\eps$ we rely on a result of Davies~~\cite{Davies3} related to
the stability of the inner perturbation $\Omega_\epsilon$.
In the sequel, we will consider bounded domains $\Omega \subset M$, where
$M$ is a complete Riemannian manifold with Ricci curvature bounded
from below by $-1$. We will also denote by $r_0$ the inradius of
$\Omega$, i.e. the radius of the largest ball included in $\Omega$.
%change ricci bound by a constant

\begin{theorem}\label{dav}
  Let $\Omega \subset M$ satisfy a weak Hardy inequality with constants
  $a, b$ and inradius bounded below by $r_0>0$. Then for each $k\in\mathbb{N}$, and each $0<\alpha \le 1/\sqrt{a},$
  there exists constants $\eps_k$ and $C_k$ depending
  only on $\alpha,a,b,r_0,k$, such that
  $$0 \le \lambda_k(\Omega_\eps)-\lambda_k(\Omega)\le C_k \eps^{2 \alpha},$$
  for any $0<\eps< \eps_k.$
\end{theorem}

\medskip
The theorem is proved in \cite{Davies3} for more general
operators, with $\Omega \subset \R^n$. In \cite[Theorem 13]{Davies2}
Davies extends the result to include the sharp exponent. For
insight we outline Davies' proof for the situation we are
considering. The proof will make use of the following result
proved in \cite[Theorem 14]{Davies3}.
 The fact that $c$ depends only on $a,\alpha$ can be deduced from
  a careful reading of the proofs  in~\cite{Davies3} and ~\cite{Davies2}.

\begin{theorem}\label{dav2}
  Let $\Omega$ satisfy a weak Hardy inequality with constants $a, b$  and
  let $0<\alpha< 1/\sqrt{a}.$ Then there exists a constant $c$
  depending only on $a,\alpha$, such
  that for each $u \in \mbox{Dom}(\Delta)$,
  $$\int_{\Omega} \frac{u^2}{\delta^{2+2\alpha}}
  \le c \lVert (-\Delta + b) u \rVert_2 \lVert (-\Delta + b)^{1/2} u
  \rVert_2,$$
  and
  $$\int_{\Omega} \frac{|\nabla u|^2}{\delta^{2\alpha}}
  \le c \lVert (-\Delta + b) u \rVert_2 \lVert (-\Delta + b)^{1/2} u \rVert_2,$$
  where $\delta$ again denotes the distance function to the boundary of $\Omega$.
\end{theorem}

\begin{proof}[Proof of Theorem \ref{dav}.]
    For each $k\in\mathbb{N}$, the span of the first $k$ eigenfunctions $\phi_i,
    i=1,2,\cdots,k$ of $\Omega$ is denoted $\Lambda_k$, and we write $\lambda_k=\lambda_k(\Omega)$ for the eigenvalues of $\Omega$. Given $u \in\Lambda_k$ with $\lVert u
    \rVert_2 =1,$
    let $0<\epsilon<r_0/2$, and write $S=\Omega\setminus \Omega_{2\eps}$.
    Define the cut-off function $\chi:\Omega\rightarrow\mathbb{R}$ by
    $$\chi(x)=
    \begin{cases}0 & \mbox{ if }x \in \Omega_{\eps}^c, \\
      \eps^{-1} \mbox{dist}(x, \partial \Omega_{\eps})  &\mbox{ if } x \in \Omega_\epsilon\setminus\Omega_{2\epsilon} ,\\
      1 &\mbox{ if } x \in \Omega_{2\eps}.
    \end{cases}$$
    Using that $\delta\leq 2\epsilon$ and $|\nabla\chi|\leq
    1/\eps$ on $S$ leads to
    \begin{align*}
      \left|
        \int_{\Omega}|\nabla(\chi u)|^2 - \int_{\Omega}|\nabla u|^2
      \right|
       & \le  \left|2\left(\int_{S}|\chi \nabla u|^2  + \int_{S}|u \nabla\chi|^2 \right) - \int_{\Omega}|\nabla u|^2 \right|\\
      & \le 2\int_{S}|u \nabla\chi|^2 + \int_{S}|\nabla u|^2\\
      & \le  \frac{2}{\eps^2} (2\eps)^{2+2\alpha} \int_{S}\frac{|u|^2}{\delta^{2+2\alpha}}
      + (2\eps)^{2\alpha} \int_{S} \frac{|\nabla u|^2}{\delta^{2\alpha}}.
    \end{align*}
    Theorem~\ref{dav2} then gives
    \begin{align*}
      \left|
        \int_{\Omega}|\nabla(\chi u)|^2 - \int_{\Omega}|\nabla u|^2
      \right|
      &
      \leq
      \eps^{2\alpha} (2^{3+2\alpha}+2^{2\alpha})c \lVert (-\Delta + b) u \rVert_2 \lVert (-\Delta + b)^{1/2} u \rVert_2 \\
      & \le  \eps^{2\alpha} 2^{2\alpha}9 c (\lambda_k + b )^{3/2} \lVert  u \rVert^2_2 = \eps^{2\alpha} 2^{2\alpha}9c (\lambda_k + b )^{3/2}.
    \end{align*}
This holds for any $0<\alpha< 1/\sqrt{a}$ with $c=c(a, \alpha).$
For the last inequality we have used that $u \in\Lambda_k$ is of the form
$u=\sum_{i=1}^k \alpha_i \phi_i$ giving $ -\Delta u =\sum_{i=1}^k \lambda_i \alpha_i \phi_i$ for some $\alpha_i \in \R,$
and thereby
\begin{gather*}
  \lVert (-\Delta + b) u \rVert_2 \le (\lambda_k + b) \lVert  u
\rVert_2,\\
\lVert (-\Delta + b)^{1/2} u \rVert_2 \le (\lambda_k + b)^{1/2} \lVert  u \rVert_2.
\end{gather*}
Next we estimate
\begin{align*}   \int_{\Omega}u^2 -\int_{\Omega}(\chi u)^2 & \le \int_{S}u^2 \le \eps^{2+2\alpha} \int_{S}\frac{|u |^2}{\delta^{2+2\alpha}} \\
& \le  \eps^{2+2\alpha} c \lVert (-\Delta + b) u \rVert_2 \lVert (-\Delta + b)^{1/2} u \rVert_2 \quad (\mbox{by Theorem \ref{dav2})}\\
& \le  \eps^{2+2\alpha} c (\lambda_k + b )^{3/2}.
\end{align*}
Combining these two estimates with the min-max principle finally gives
\begin{align*}
  \lambda_k(\Omega_{\eps}) & \le
  \sup\left\{\frac{\lVert \nabla v \rVert^2_2}{\lVert  v \rVert^2_2} :
    v \in
    \mbox{span}\{\chi \phi_1, \cdots , \chi \phi_k\} \right\}\\
  &=\sup\left\{\frac{\lVert \nabla (\chi u) \rVert^2_2}{\lVert \chi u \rVert^2_2} : u \in\Lambda_k \right\}
\le \frac{\lambda_k + \eps^{2\alpha} 2^{2\alpha}9 c (\lambda_k + b )^{3/2}}{1- \eps^{2+2\alpha} c (\lambda_k + b )^{3/2}}.
\end{align*}
If we take
$$
\eps \le \left(\frac{1}{2c (\lambda_k + b )^{3/2}}\right)^{1/(2+2\alpha)}:=\eps_k,
$$
we have
\begin{align*}
\lambda_k(\Omega_{\eps})-\lambda_k\le \frac{\eps^{2\alpha} 2^{2\alpha}9 c(\lambda_k + b )^{3/2}}{1- \eps^{2+2\alpha} c (\lambda_k + b )^{3/2}}
\le  \eps^{2\alpha} 2^{2\alpha+1}9 c(\lambda_k + b )^{3/2}.
\end{align*}
Using again that the inradius is bounded below by $r_0$, it follows
from the proof of~\cite[Corollary 2.3]{Cheng}
that there exists a constant $\tilde{c}$ depending only on $n$ and $k$
such that
$$
 \lambda_k \le \frac{(n-1)^2}{4}+\frac{\tilde{c}}{r_0^2}.
$$
Putting everything together, we have
\begin{align*}
 \lambda_k(\Omega_{\eps})-\lambda_k \le C_k(\alpha, a,b,r_0)\eps^{2\alpha},
\end{align*}
for all $\eps \le \eps_k=\eps_k(\alpha,a,b,r_0)$. For the
extension to the sharp  exponent $\frac{2}{\sqrt{a}}$ please see
\cite[Theorem 13]{Davies2}.
\end{proof}

\subsection{Examples of locally stable families}
\label{SubsectionFamilies}
In this section we construct families $\mathcal{A}$ of
domains such that for each $k$ and each $\epsilon>0$ small enough, the
difference
$|\lambda_k(\Omega^\epsilon)-\lambda_k(\Omega)|$ is
controlled in terms of $\epsilon$, $k$ and some geometric hypotheses
on $\Omega$. The bounds are uniform in
$\Omega\in\mathcal{A}$.

The goal is to describe families of domains for which we can use
Theorem \ref{dav} to get such uniform estimates.
Because Theorem \ref{dav} relates $\lambda_k(\Omega)$ to
$\lambda_k(\Omega_{\eps})$ rather than to $\lambda_k(\Omega^{\eps})$, we
will need the following conditions to hold for some $\eps_0>0$ in
addition to the lower bound $r_0$ on the inradius.
\begin{itemize}
\item There exists $N>0$ such that
$(\Omega^\eps)_{N\eps}$ is contained in $\Omega$ for each $\eps\le\eps_0$.
\item The sets $\Omega^\eps$ satisfy the weak Hardy inequality
  with constants $a,b$ independent of $\eps\le\eps_0$.
\end{itemize}

This allows control of
$|\lambda_k(\Omega^\eps)-\lambda_k((\Omega^\eps)_{N\eps})|$, and so of
$|\lambda_k(\Omega^{\eps}) -\lambda_k(\Omega)|$ because $(\Omega^\eps)_{N\eps} \subset \Omega$.
In the remainder of this section, we give an exposition of the situation in
Euclidean space $\mathbb{R}^n$.

\medskip
The following three definitions will be used.
\begin{definition}
An open set $\Omega \subset \R^n$ satisfies a \emph{uniform external
  cone condition} ~\cite[p.129]{Davies} with parameters $\alpha >0, \beta>0$
if for any $x\in \Omega^c$ we can find a cone $C$ of height $\beta$ and angle $\alpha$ with $x\in C$ and $C \subset \Omega^c.$
\end{definition}

\begin{definition}
An open set $\Omega \subset \R^n$ satisfies a \emph{uniform external
  ball condition} ~\cite[p. 27]{Davies} with parameters $\alpha >0, \beta>0$
if for any $z \in \partial \Omega$ and $0<r\le \beta$ there exists $x\in \Omega^c$ with $d(z,x)\le r$
such that $B(x, \alpha r ) \subset \Omega^c.$
It satisfies a \emph{uniform external rolling ball condition} with parameter $\beta>0$ if this holds with $\alpha=1$.
\end{definition}

\begin{definition}
An open set $\Omega \subset \R^n, n\ge 3$ satisfies a \emph{uniform capacity
  density condition} ~\cite[Lemma 3]{Ancona} with parameter $\alpha >0$
if for any $z \in \partial \Omega$ and any $r > 0$ $$\mbox{cap}\,(B(z,  r )\setminus \Omega) \ge \alpha r^{n-2}.$$
Here $\mbox{cap}\,$ is the capacity defined by
\begin{equation*}
\mbox{cap}\,(\Gamma)=\inf \left\{ \int_{\R^n}(|\nabla v|^2+v^2): v \in
C_0^{\infty}(\R^n),\ v\ge 1\ \textup{on}\ \Gamma \right \}.
\end{equation*}
\end{definition}

\begin{remark}
The uniform capacity density condition is weaker than the uniform external ball condition,
which again is implied by the uniform external cone condition and finally by the stronger uniform Lipschitz condition. All of these conditions imply that the set satisfies a weak Hardy inequality.
\end{remark}

\subsubsection{Uniform external rolling ball condition}

\begin{proposition}\label{Proposition:Rolling}
Let $\eps_0>0, r_0>0$. Let $\mathcal A= \mathcal A(\eps_0,r_0,n)$
be a family of open, bounded sets $\Omega$ in Euclidean space $\R^n,
\; n\ge 2,$ with inradius bounded below by $r_0$. Then each $\Omega \in
\mathcal A$ satisfies the uniform external rolling ball condition
with parameter
 $\beta = \eps_0$ if and only if
$(\Omega^\eps)_\eps \subset \Omega$ for all $\eps \le \eps_0$.
In this case the sets $\Omega^\eps, \ \eps \le \eps_0/2$, satisfy the uniform external rolling ball condition with parameter
$\beta = \eps_0/2$.
\end{proposition}
\begin{proof} Suppose  $(\Omega^\eps)_\eps \subset \Omega$ for all $\eps \le \eps_0$.
Let $z \in \partial \Omega$ and $0<\epsilon\le \eps_0.$ Then as $(\Omega^\epsilon)_\epsilon \subset \Omega$, $z$ is not in  $(\Omega^\epsilon)_\epsilon$,
and so $d(z, (\Omega^\epsilon)^c) = \epsilon$.
This gives $\overline{B(z,\epsilon)} \cap (\Omega^\epsilon)^c \neq \emptyset, $ and so we choose $x$ in this set.
Then  $x \in (\Omega^\epsilon)^c$ implies $d(x,\Omega)\ge \epsilon$, and so we have $B(x, \epsilon ) \subset \Omega^c$ with $d(z,x)\le \epsilon$
as required.

Now suppose $\Omega$
satisfies the uniform external rolling ball condition with parameter $\beta = \eps_0$.
Let $z \in \Omega^c$ and $0<\epsilon\le \eps_0.$ By the rolling ball condition, there exists $x$ such that
$d(z,x)\le \eps_0$ and $B(x, \eps_0 ) \subset \Omega^c.$ As $\eps\le\eps_0$, we have $x \in (\Omega^{\eps})^c$ and so
$d(z,(\Omega^{\eps})^c)\le \eps$. Thus $z$ is not in $(\Omega^{\eps})_{\eps}$, and we conclude that $(\Omega^\eps)_\eps \subset \Omega$.

Now note that if $(\Omega^\eps)_\eps \subset \Omega$ for all $\eps \le \eps_0$, then the sets $\Omega^\eps, \ \eps \le \eps_0/2$,
satisfy a condition of the form  $((\Omega^\eps)^\delta)_\delta \subset \Omega^\eps$ for all $\delta \le \eps_0/2$.
Hence they also satisfy the uniform external rolling ball condition with parameter
$\beta = \eps_0/2$.
\end{proof}

\begin{lemma}\label{Lemma:externalBall} Let $\eps_0>0, r_0>0$.
  Let $\mathcal A= \mathcal A(\eps_0,r_0,n)$ be the family of open, bounded sets $\Omega$
  in  $\R^n, \; n\ge 2,$ with inradius bounded below by
  $r_0$ satisfying a uniform external rolling ball condition with parameter $\eps_0$.
  Then there exist constants $\gamma=\gamma(n)$, $
  \eps_k=\eps_k(n, \eps_0,r_0)$ and $C_k=C_k(n,\eps_0,r_0)$ such that
  for any $\Omega \in \mathcal A$ and any $\eps \le min(\eps_0/2,\eps_k),$
  \begin{equation}
    0 < \lambda_k(\Omega)-\lambda_k(\Omega^\eps) \le C_k \eps^{\gamma}.
  \end{equation}
\end{lemma}

\begin{proof}
  For $\epsilon<\epsilon_0/2$, the sets $\Omega^\epsilon$ satisfy the
  rolling ball condition with parameter $\epsilon_0/2$ by Proposition \ref{Proposition:Rolling}.
  By the proof of Theorem 1.5.4 in \cite{Davies}, the sets satisfy the weak Hardy inequalities

\begin{align*}
  \int_{\Omega} \frac{u^2}{\delta^2} &\le
  a \left(\int_{\Omega}|\nabla u|^2 + b u^2 \right), \forall u \in C^{\infty}_0(\Omega), \\
  \int_{\Omega^\eps} \frac{u^2}{\delta^2}
  & \le a \left(\int_{\Omega^\eps}|\nabla u|^2 + b u^2 \right), \forall u \in C^{\infty}_0(\Omega^\eps), \forall \eps \le \eps_0/2,
\end{align*}
with constants $b=(2/\epsilon_0)^2$ and
 $$a = \frac{n}{32}\int_0^{\pi/6}\sin^{n-2}(t)dt / \int_0^{\pi/2} \sin^{n-2}(t)dt.$$
 Here $\delta$ denotes the distance function to the boundary.
 Then by Theorem \ref{dav}, \eqref{e1} follows with $\gamma \le \frac{2}{\sqrt{a}}$.
\end{proof}

Combining Lemma~\ref{external} with
Proposition~\ref{Proposition:Rolling} leads to the following
corollary.

\begin{cor} \label{external} Let $\eps_0>0, r_0>0$.
Let $\mathcal A= \mathcal A(\eps_0,r_0,n)$ be the family of open, bounded sets $\Omega$
in  $\R^n, \; n\ge 2,$ with inradius bounded below by $r_0$ and such that for each $\Omega\in \mathcal A$,
$(\Omega^\eps)_\eps \subset \Omega$ for all $\eps \le \eps_0$.
Then there exist constants $\gamma=\gamma(n)$, $
  \eps_k=\eps_k(n, \eps_0,r_0)$ and $C_k=C_k(n,\eps_0,r_0)$ such that
  for any $\Omega \in \mathcal A$ and any $\eps \le min(\eps_0/2,\eps_k),$
  \begin{equation}
    0 < \lambda_k(\Omega)-\lambda_k(\Omega^\eps) \le C_k \eps^{\gamma}.
  \end{equation}
\end{cor}

\subsubsection{Capacity density condition.}

\begin{lemma} Let $N, \eps_0, r_0>0$.
Let $\mathcal A= \mathcal A(N,\eps_0,r_0,\alpha,n)$ be the family of open, bounded sets $\Omega$
in $\R^n, \; n\ge 3,$ such that $(\Omega^\eps)_{N\eps} \subset \Omega$ for all $\eps \le \eps_0$.
Suppose also $\Omega^\eps$ satisfies a uniform capacity density condition with parameter $\alpha>0$, uniformly in $\eps \le \eps_0$.
Then there exist constants $\gamma=\gamma(\alpha)$, $\eps_k=\eps_k(n,r_0,\alpha)$ and $C_k=C_k(n,r_0,\alpha)$ such that
for any $\Omega \in \mathcal A$ and any $\eps \le min(\eps_0,\eps_k),$
\begin{equation} \label{e5}
0 < \lambda_k(\Omega)-\lambda_k(\Omega^\eps) \le C_k(N \eps)^{\gamma}.
\end{equation}

\end{lemma}

This follows by a combination of~\cite[Theorem 4.2]{Davies3} with
Proposition 1 and ~\cite[Lemma 3]{Ancona}, which gives a Hardy inequality
for sets satisfying a uniform capacity density condition.

\subsubsection{Uniform external cone condition}
\begin{lemma} Let $\eps_0>0$, $r_0>0$.
Let $\mathcal A= \mathcal A(\eps_0,r_0,\alpha,\beta,n)$ be the family of open, bounded sets
in  $\R^n, \; n\ge 2,$ with inradius bounded below by
$r_0$ and such that for each $\Omega\in \mathcal A$, the sets $\Omega,
\Omega^\eps, \eps \le \eps_0$ satisfy a uniform external cone condition
with parameters $\alpha >0, \beta>0$.
Then there exist constants $\gamma=\gamma(\alpha, \beta)$,
$\eps_k=\eps_k(n,\alpha,\beta,r_0)$ and $C_k=C_k(n,\alpha, \beta,r_0)$
such that for any $\eps \le \min(\frac{\beta \tan(\alpha/2)}{\tan(\alpha/2)+1},\eps_k),$
and any $\Omega\in\mathcal{A}$,
\begin{equation} \label{e1}
  0 < \lambda_k(\Omega)-\lambda_k(\Omega^\eps) \le C_k \eps^{\gamma}.
\end{equation}

\end{lemma}

\begin{proof}
As mentioned ~\cite[Lemma 3]{Ancona} gives a Hardy inequality
for sets satisfying a uniform capacity density condition and hence also for sets satisfying the stronger
 uniform external cone condition. This ensures that a weak Hardy inequality exists uniformly
for the sets $\Omega, \Omega^\eps, \eps \le \eps_0$.
We now show $(\Omega^\eps)_{N\eps} \subset \Omega$ for some $N=N(\alpha)$ uniformly in $\eps\le\frac{\beta \tan(\alpha/2)}{\tan(\alpha/2)+1}=:\eps^*$
by showing that $x \in \Omega^c$ implies
$x \in ((\Omega^{\eps^*})_{N\eps^*})^c.$ Then \eqref{e1} follows by Theorem \ref{dav}.

Take $x \in \Omega^c$ and let $C \subset \Omega^c$ be a cone of height $\beta$ and angle $\alpha$ with $x \in C$.
Such a cone contains a ball of radius $\eps^*$ centered at a point $y$ for which we then have
$y \in (\Omega^{\eps^*})^c$ and $d(x,y)\le \frac{\eps^*}{\tan(\alpha/2)}$.
Thus $x \in ((\Omega^{\eps^*})_{N\eps^*})^c$ for $N \ge (\tan(\alpha/2))^{-1}$.
\end{proof}

\subsubsection{The case of $\R^2$}

\begin{lemma} Let $N, \eps_0,r_0>0$.
Let $\mathcal A= \mathcal A(N,\eps_0,r_0)$ be the family of open, bounded sets $\Omega$
in $\R^2$  such that $\Omega^\eps$ is simply connected with inradius
bounded below by $r_0$, and satisfies $(\Omega^\eps)_{N\eps} \subset \Omega$ for all $\eps \le \eps_0$.
Then there exist constants $C_k=C_k(r_0)$ and $\eps_k=\eps_k(r_0)$ such that
for any $\Omega \in \mathcal A$ and any $\eps \le min(\eps_0,\eps_k),$
\begin{equation} \label{e4}
  0 < \lambda_k(\Omega)-\lambda_k(\Omega^\eps) \le C_k \eps^{1/2}.
\end{equation}
\end{lemma}

The inequality  follows by Theorem \ref{dav} using that simply
connected planar sets satisfy a Hardy inequality with $a=16, b=0$ ~\cite[Theorem 1.5.10]{Davies}.
Note that the  sharp exponent for Theorem \ref{dav} in $\R^2$ was
first given in \cite{Pang}.

\section{Proximity of eigenspaces}
\label{SectionProximityEigenspace}
\subsubsection*{Some notation}
Denote by $\lambda_i=\lambda_i(\Omega)$ the eigenvalues of $\Omega$,
and by $\lambda_i' = \lambda_i(\Omega')$ the eigenvalues of $\Omega'$.
Most of this paper has been devoted to estimation of the difference
$\lambda_i-\lambda_i'.$
We write
$$
\delta_i=max_{j=1}^{i}(\lambda_j-\lambda_j').
$$

We consider an orthonormal basis
$(f_i')_{i=1}^{\infty}$ of eigenfunctions of $\Omega'$ corresponding to
the eigenvalues $\lambda_i'$.
For $k\geq 1$, let $E_k$ be the successive distinct eigenspaces
for $\Omega$. Let $n_k=\dim (E_k)$ be the corresponding
multiplicities and set $N_0=0$ and
$N_k=n_1+...+n_k$.
In particular, this means that for each $k\in\mathbb{N}$,
$$\lambda_{N_{k-1}+1}=...=\lambda_{N_{k-1}+n_k}=\lambda_{N_k}.$$
We also introduce the corresponding
vector space $E_k'$ spanned by the eigenfunctions
$f_{N_{k-1}+1}',...,f_{N_k}'$ on $\Omega'$.
The projection operator $P_k:L^2(\Omega')\rightarrow E_k'$ is defined,
for $f=\sum_{j=1}^{\infty} a_jf_j'$,
by
$$P_k(f) =  \sum_{j=N_{k-1}+1}^{N_k} a_jf_j'.$$
The gaps between successive distinct eigenvalues of $\Omega$ are
$\lambda_{N_k+1}-\lambda_{N_k},$
and we write
$$
\Lambda_k= min_{j=1}^k (\lambda_{N_j+1}-\lambda_{N_j}).
$$
In particular, if $i \le k$, we have the relation $\frac{\delta_i}{\Lambda_i} \le \frac{\delta_k}{\Lambda_k}$
which will be used in the proof, and $\frac{\delta_k}{\Lambda_k}\le \frac{1}{2}$ implies
$\frac{\delta_i}{\Lambda_i}\le \frac{1}{2}$ for all $i \le k$.

\medskip
We can now express the proximity of eigenspaces.
\begin{proposition} \label{proximity}
  There exists a sequence of constants $(A_k)_{k=1}^{\infty}$ such that
  if $\delta_{N_k+1} \le \frac{\Lambda_k}{2 A_k}$, then the
    following holds:
  \begin{enumerate}
  \item[A.]
    For each $f\in E_k$, $\Vert f\Vert=1$, we have
    $$\Vert (I-P_k)(f) \Vert^2 =1 -\Vert P_k \Vert^2 \le \frac{A_k\delta_{N_{k-1}+1}}{\Lambda_k}.$$

  \item[B.]
    For each $f' \in E_k'$, $\Vert f' \Vert=1$, there exist $f \in E_k$, $\Vert f \Vert=1$ with
    $$
    \Vert f-f'\Vert^2 \le \frac{4A_k \delta_{N_{k-1}+1}}{\Lambda_k}.
    $$

  \item[C.]
    For each $f\in E_{k+l}$, $l>0$, $\Vert f\Vert=1$, we have
    $$
    \Vert (P_1+...+P_k)(f)\Vert^2 \le \frac{4 (\sum_{i=1}^kA_i) \delta_{N_{k-1}+1}}{\Lambda_k}.
    $$
\end{enumerate}

Moreover, the constants are defined by the induction
$A_k= 2+ 8\frac{\lambda_{N_{k-1}+1}\sum_{i=1}^{k-1}A_{i}}{\Lambda_k}$,
and so depend on $k$ and  the spectrum of $\Omega$.

\end{proposition}

We first show the following technical proposition.
\begin{proposition} \label{calculation} Under the condition
$
\delta_{N_i+1} \le \frac{1}{2} \Lambda_i,
$
we have, for $f \in E_i$, $\Vert f\Vert=1$,
$$
\Vert f- P_i(f)\Vert^2=1-\Vert P_i(f)\Vert^2 \le \frac{2}{\Lambda_i}(\delta_{N_{i-1}+1}+ \lambda_{N_i+1}'
\Vert P_1(f)+...+P_{i-1}(f)\Vert^2).
$$
\end{proposition}

\noindent
\begin{proof}
we consider $f \in E_i$ and, applying $q$, we get the relations
\begin{align*}
\lambda_{N_{i-1}+1} - q(P_i(f))=
\sum_{j \not = i}q(P_j(f))\ge \sum_{j<i}q(P_j(f))+\lambda'_{N_i+1} (1-\sum_{j \le i}\Vert P_j(f)\Vert^2),
\end{align*}
and 
\begin{align*}
\lambda_{N_{i-1}+1} - q(P_i(f))=
 \lambda_{N_{i-1}+1}-\lambda'_{N_{i-1}+1} +
 (1-\Vert P_i(f)\Vert^2)\lambda'_{N_{i-1}+1} -(q(P_i(f))-\lambda'_{N_{i-1}+1}\Vert P_i(f)\Vert^2).
 \end{align*}
Putting this together, we get
\begin{gather*}
 \lambda_{N_{i-1}+1}-\lambda'_{N_{i-1}+1} +  (1-\Vert P_i(f)\Vert^2)\lambda'_{N_{i-1}+1}\\
  \ge
  \sum_{j<i}q(P_j(f))+\lambda'_{N_i+1}(1- \sum_{j\le i}\Vert P_j(f)\Vert^2)+
  q(P_i(f))-\lambda'_{N_{i-1}+1}\Vert P_i(f)\Vert^2.
\end{gather*}
Noticing that  $q(P_i(f))-\lambda'_{N_{i-1}+1}\Vert P_i(f)\Vert^2 \ge 0$, we get
\begin{gather*}
 \lambda_{N_{i-1}+1}-\lambda'_{N_{i-1}+1} +\lambda'_{N_i+1} \sum_{j< i}\Vert P_j(f)\Vert^2
 \ge (1-\Vert P_i(f)\Vert^2)(\lambda'_{N_i+1}-\lambda'_{N_{i-1}+1}),
\end{gather*}
i.e.
\begin{eqnarray*}
(1-\Vert P_i(f)\Vert^2) \le
 \frac{(\lambda_{N_{i-1}+1}-\lambda'_{N_{i-1}+1})+\lambda'_{N_i+1} \sum_{j< i}\Vert P_j(f)\Vert^2 }
 {\lambda'_{N_i+1}-\lambda'_{N_{i-1}+1}}.
\end{eqnarray*}
Moreover,
\begin{align*}
  \lambda'_{N_i+1} -\lambda'_{N_{i-1}+1}
  &=
  (\lambda'_{N_i+1}-\lambda_{N_i+1})+\\
  (\lambda_{N_i+1}-&\lambda_{N_{i-1}+1})+(\lambda_{N_{i-1}+1}-\lambda'_{N_{i-1}+1})
  \ge \Lambda_i- (\lambda_{N_i+1}-\lambda'_{N_i+1}).
\end{align*}
So under the condition
$
(\lambda_{N_i+1}-\lambda'_{N_i+1}) \le \frac{1}{2} \Lambda_i,
$
we get
\begin{eqnarray*}
(1-\Vert P_i(f)\Vert^2) \le
 2 \frac{(\lambda_{N_{i-1}+1}-\lambda'_{N_{i-1}+1})+\lambda'_{N_i+1} \sum_{j< i}\Vert P_j(f)\Vert^2 }
 {\Lambda_i},
\end{eqnarray*}
\end{proof}

\medskip
Before proving Proposition \ref{proximity}, we recall a fact of linear algebra which is needed for the proof.

\begin{lemma} \label{lemme} Let $(V,\left\langle.,.\right\rangle)$ be a prehilbert vector space, $P$ a projector and $v$ a vector of norm $1$ such that $P(v) \not = 0$. Then we have
$$
\Vert v-\frac{P(v)}{\Vert P(v)\Vert} \Vert^2 \le 4(1-\Vert P(v)\Vert^2).
$$
\end{lemma}

\begin{proof} We have
$$
\Vert v-\frac{P(v)}{\Vert P(v)\Vert} \Vert \le \Vert v-P(v) \Vert+\Vert P(v)- \frac{P(v)}{\Vert P(v)\Vert}\Vert
\le 2 \Vert v-P(v)\Vert
$$
 and
$$
 \Vert v-P(v) \Vert^2 = \langle v,v-P(v)\rangle=\langle v,v\rangle-\langle v,P(v)\rangle
 =1-\Vert P(v)\Vert^2.
$$
\end{proof}

\medskip
\noindent
\begin{proof}[Proof of Proposition \ref{proximity}]
The proof is by induction on $k$:
\begin{enumerate}
\item
We show that A is true for $k=1$.

\item
We show that if A is true for $k$, then B is true for $k$.

\item
We show that if A and B are true for $1,...,k$, then C is true for $k$.

\item
We show that if A, B, C are true for $1,...,k-1$, then A is true for $k$.
\end{enumerate}

So because A is true for $k=1$, it follows that B and C are true for $k=1$, and this implies that
A is true for $k=2$. Then the induction continues in the obvious way.

\medskip
\noindent
(1) Proof of A for $k=1$: this follows directly from Proposition \ref{calculation} with $A_1=2$.

\medskip
\noindent
(2) Proof of A true for k implies B for k:
let $f'\in E_k'$, $\Vert f'\Vert =1$. If
$\delta_{N_{k-1}+1}\le\frac{\Lambda_k}{2A_k}$, the restriction of
$P_k$ to $E_k$ is bijective and there exist
$f\in E_k$, $\Vert f\Vert =1$ with
$f'= \frac{P_k(f)}{\Vert P_k(f)\Vert}$. It follows from
Lemma~\ref{lemme} that
$\Vert f'-f\Vert ^2 \le 4\Vert (I-P_k)(f)\Vert^2 \le 4  \frac{A_k\delta_{N_{k-1}+1}}{\Lambda_k}$.

\medskip
\noindent
(3) Proof of A and B true for $1,...,k$ implies C for $k$.

\medskip
\noindent
Let $f \in E_{k+l}$, $l\ge 1$ and $\Vert f\Vert=1$. Let $i \le k$. For
any $h'\in E_i'$, $\Vert h'\Vert=1$, because $P_i$ is surjective under
the hypothesis $\delta_{N_i+1} \le \frac{\Lambda_i}{2 A_i}$, there
exist $h\in E_i$, $\Vert h\Vert =1$ with $h'=\frac{P_i(h)}{\Vert P_i(h)\Vert}$.
We have
$$
\left\langle P_i(f),h'\right\rangle= \left\langle f,h'\right\rangle=
\left\langle f,h'-h\right\rangle+\left\langle f,h\right\rangle.
$$
As in the case $k=1$, we have $\left\langle f,h\right\rangle=0$, so
$$
\left\langle P_i(f),h'\right\rangle =\left\langle f,h'-h\right\rangle \le \Vert f\Vert
\Vert h'-h\Vert= \Vert h-\frac{P_i(h)}{\Vert P_i(h)\Vert}\Vert.
$$
By Lemma \ref{lemme}, we get
$\vert \left\langle P_i(f),h'\right\rangle \vert ^2\le 4 (1-\Vert P_i(h)\Vert^2) \le 4A_i\frac{\delta_{N_{i-1}+1}}{\Lambda_i}$ by
A.
Using  $\frac{\delta_{N_{i-1}+1}}{\Lambda_i}\le \frac{\delta_{N_{k-1}+1}}{\Lambda_k}$,
and because this is true for each $h' \in E_i'$, $\Vert h'\Vert =1$, we have
$$
\Vert P_i(f)\Vert^2 \le 4A_i \frac{\delta_{N_{k-1}+1}}{\Lambda_k},
$$
and we deduce that
$$
\sum_{i=1}^k \Vert P_i(f)\Vert^2 \le 4(\sum_{i=1}^k A_i)\frac{\delta_{N_{k-1}+1}}{\Lambda_k}.
$$

\medskip
\noindent
(4) Proof of A, B, C true for $1,...,k-1$ implies A for $k$.

\medskip
\noindent
Note that, because we have $A_i\ge 2$, the hypothesis $\delta_{N_k+1} \le \frac{\Lambda_k}{2 A_k}$ implies
$\delta_{N_k+1} \le \frac{\Lambda_k}{2}$.
By Proposition \ref{calculation}, we have for $f \in E_k$, $\Vert f\Vert=1$, and under the condition $\delta_{N_k+1} \le \frac{\Lambda_k}{2}$,
$$
\Vert (I-P_k)(f)\Vert^2 \le \frac{2}{\Lambda_i}(\delta_{N_{k-1}+1}+ \lambda_{N_k+1}'
\Vert P_1(f)+...+P_{k-1}(f)\Vert^2).
$$
By induction, $\Vert P_1(f)+...+P_{k-1}(f)\Vert^2 \le 4 (\sum_{i=1}^{k-1}A_{i}) \frac{\delta_{N_{k-2}+1}}{\Lambda_{k-1}}$, which is by definition $\le 4 (\sum_{i=1}^{k-1}A_{i}) \frac{\delta_{N_{k-1}+1}}{\Lambda_{k}}$.

\smallskip
So we choose $A_k= 2+ 8\frac{\lambda_{N_{k-1}+1}\sum_{i=1}^{k-1}A_{i}}{\Lambda_k}$.
\end{proof}

\section{Proof of Lemma~\ref{LemmaBertrandColbois}}
\label{SecionProofBertrandColbois}
The proof of Lemma ~\ref{LemmaBertrandColbois} is very close to the proof of ~\cite[Lemma 3.13]{BertrandColbois}. Under the hypothesis
$$
\vert \left\langle \psi_i,\psi_j\right\rangle -\delta_{i,j}\vert \le \rho;\ \ q(\psi_i)\le \lambda_i(1+\rho),
$$
for $1 \le i,j \le k$, we prove that
$$
q(F_i) \le \lambda_i(1+\rho b_k).
$$

Here the constants $b_k$ are defined by $b_k=(1+8a_k)(1+(1+\rho b_{k-1})(1+8a_k)) $ (where the $a_k$ are defined in the statement of the lemma) with $b_1=4$,
and $\{F_i\}_{i=1}^{k}$ is the orthonormal basis naturally associated to the basis $\{\psi_i\}_{i=1}^{k}$. We have
$$
F_i= \frac{h_i}{\Vert h_i\Vert}\ \text{where}\ h_i=\psi_i-\sum_{j=1}^{i-1}\left\langle F_j,\psi_i\right\rangle F_j.
$$
In particular, $b_k$ depends only on $k$ (and not on $\lambda_k$), but in a rather complicated form.
We also denote by $p(\psi_i)$ the projection of $\psi_i$ given by
$
p(\psi_i)=\sum_{j=1}^{i-1}\left\langle F_j,\psi_i\right\rangle F_j,
$
so that we have the relation $\psi_i=h_i+p(\psi_i).$

\medskip
\noindent
\begin{proof}

\medskip
\noindent
\textbf{First part:} The beginning of the proof is verbatim the same as the proof of ~\cite[Lemma 3.13]{BertrandColbois}.

\medskip
We suppose the result is true
for $s<k$ and show it for $s=k$. As a first step, note that
$$
q(F_1) =\frac{q(\psi_1)}{\Vert \psi_1\Vert^2} \le \lambda_1 \frac{1+\rho}{1-\rho}=\lambda_1(1+\rho \frac{2}{1-\rho}),
$$
so using $\rho<\frac{1}{2}$ we can take $b_1=4$.
We also use without repeating the proof the following fact proved in the first part of ~\cite[Lemma 3.13]{BertrandColbois}.
For $i=1,...,k-1$ and $s>i$, we have
$$
\vert \left\langle \psi_s,F_i\right\rangle \vert \le \sqrt 2 a_i\rho.
$$
We have for $s <k$
$$
q(p(\psi_s))= \sum_{j=1}^{s-1}\sum_{l=1}^{s-1} \left\langle \psi_s,F_j\right\rangle \left\langle \psi_s,F_l\right\rangle q(F_j,F_l).
$$
By the recurrence hypothesis we have
$$
q(F_j,F_l)\le q^{1/2}(F_j)q^{1/2}(F_l) \le (\lambda_j(1+\rho b_j))^{1/2} (\lambda_l(1+\rho b_l))^{1/2}
\le  \lambda_{s-1}(1+\rho b_{s-1}),
$$
and, using the definition of the $a_j$ along with the Cauchy-Schwarz inequality,
$$
q(p(\psi_s))\le \lambda_{s-1}(1+\rho b_{s-1})  2\rho^2 \sum_{j,l=1}^{s-1} a_ja_l \le 2\rho^2 a_s\lambda_{s-1}(1+\rho b_{s-1}).
$$
We will also use the fact that $\Vert h_s\Vert^2 \ge 1-\rho a_s$. This implies (and this is the main change)
$$
\frac{1}{\Vert h_s\Vert^2} \le 1+2\rho a_s.
$$

\medskip
\noindent
\textbf{Second part:} The following includes some new developments in comparison with ~\cite[Lemma 3.13]{BertrandColbois}.

\medskip
We use that for $2\le s \le k$
$$
q(F_s) \le \frac{1}{\Vert h_s\Vert^2}(q(\psi_s)+q(p(\psi_s))+ 2\sqrt{q(\psi_s)q(p(\psi_s))}
$$
and get the following estimate for $q(F_k)$:
$$
q(F_k) \le (1+2\rho a_k)[\lambda_k(1+\rho)+\lambda_{k-1}2\rho^2a_k(1+\rho b_{k-1})+
2 \lambda_k^{1/2}\lambda_{k-1}^{1/2}(1+\rho)^{1/2}(2a_k(1+\rho b_{k-1}))^{1/2}\rho].
$$
Because $\lambda_{k-1} \le \lambda_k$, we get
$$
q(F_k)\le \lambda_k(1+2\rho a_k)[(1+\rho)+2\rho^2a_k(1+\rho b_{k-1})+
2 (1+\rho)^{1/2}(2a_k(1+\rho b_{k-1}))^{1/2}\rho].
$$
But
$$
(1+2\rho a_k)[(1+\rho)+2\rho^2a_k(1+\rho b_{k-1})+
2 (1+\rho)^{1/2}(2a_k(1+\rho b_{k-1}))^{1/2}\rho]=
$$
$$
1+\rho 2a_k[(1+\rho)+2\rho^2a_k(1+\rho b_{k-1})+
2 (1+\rho)^{1/2}(2a_k(1+\rho b_{k-1}))^{1/2}\rho]+
$$
$$
+\rho[(1+2\rho a_k(1+\rho b_{k-1})+
2 (1+\rho)^{1/2}(2a_k(1+\rho b_{k-1}))^{1/2}]
$$
So to obtain the conclusion, we have to show that
$$
b_k \ge 2a_k[(1+\rho)+2\rho^2a_k(1+\rho b_{k-1})+
2 (1+\rho)^{1/2}(2a_k(1+\rho b_{k-1}))^{1/2}\rho]+
$$
$$
+[(1+2\rho a_k(1+\rho\gamma_{k-1})+2 (1+\rho)^{1/2}(2a_k(1+\rho b_{k-1}))^{1/2}].
$$
We simplify this expression using inequalities implied by the definition of $\rho$ and $a_k$ such as $\rho <1/2$, $2\rho^2a_k<1$, $2\rho a_k \le 1$ along with 
$\sqrt x \le x$ if $x\ge 1$.
This gives
$$
[(1+\rho)+2\rho^2a_k(1+\rho b_{k-1})+
2 (1+\rho)^{1/2}(2a_k(1+\rho b_{k-1}))^{1/2}\rho]\le
$$
$$
\le
[(1+\rho)+2\rho^2a_k(1+\rho b_{k-1})+
2 (1+\rho)(2a_k(1+\rho b_{k-1}))\rho]\le
$$
$$
\le [(1+\rho)+(1+\rho b_{k-1})(2\rho^2a_k+4(1+\rho)a_k)]\le 2+(1+\rho b_{k-1})(1+8a_k).
$$

$$
[(1+2\rho a_k(1+\rho b_{k-1})+2 (1+\rho)^{1/2}(2a_k(1+\rho b_{k-1}))^{1/2}]\le
$$
$$
\le ([(1+2\rho a_k(1+\rho b_{k-1})+2 (1+\rho)(2a_k(1+\rho b_{k-1}))]=
$$
$$
=1+(1+\rho b_{k-1})[2\rho a_k+4(1+\rho)a_k]\le 1+(1+8a_k)(1+\rho b_{k-1}).
$$
We can then finish with
$$
2a_k [2+(1+\rho b_{k-1})(1+8a_k)]+[1+(1+8a_k)(1+\rho b_{k-1})]=
(1+\rho b_{k-1})(1+8a_k)(1+2a_k)+1+4a_k \le
$$
$$
\le (1+\rho b_{k-1})(1+8a_k)(1+8a_k)+1+8a_k=
(1+8a_k)(1+(1+\rho b_{k-1})(1+8a_k))=b_k.
$$
\end{proof}

\bibliographystyle{plain}
\bibliography{biblioCGI2}

\def\cprime{$'$} \def\cprime{$'$}
\begin{thebibliography}{10}

\bibitem{Ancona}
Alano Ancona.
\newblock On strong barriers and an inequality of {H}ardy for domains in {${\bf
  R}^n$}.
\newblock {\em J. London Math. Soc. (2)}, 34(2):274--290, 1986.

\bibitem{BertrandColbois}
J.~Bertrand and B.~Colbois.
\newblock Capacit\'e et in\'egalit\'e de {F}aber-{K}rahn dans {${\Bbb R}^n$}.
\newblock {\em J. Funct. Anal.}, 232(1):1--28, 2006.

\bibitem{BLL}
V.~I. Burenkov, P.~D. Lamberti, and M.~Lantsa~de Kristoforis.
\newblock Spectral stability of nonnegative selfadjoint operators.
\newblock {\em Sovrem. Mat. Fundam. Napravl.}, 15:76--111, 2006.

\bibitem{Cheng}
Shiu~Yuen Cheng.
\newblock Eigenvalue comparison theorems and its geometric applications.
\newblock {\em Math. Z.}, 143(3):289--297, 1975.

\bibitem{DanerJDE}
Daniel Daners.
\newblock Dirichlet problems on varying domains.
\newblock {\em J. Differential Equations}, 188(2):591--624, 2003.

\bibitem{Davies}
E.~B. Davies.
\newblock {\em Heat kernels and spectral theory}, volume~92 of {\em Cambridge
  Tracts in Mathematics}.
\newblock Cambridge University Press, Cambridge, 1989.

\bibitem{Davies3}
E.~B. Davies.
\newblock Eigenvalue stability bounds via weighted {S}obolev spaces.
\newblock {\em Math. Z.}, 214(3):357--371, 1993.

\bibitem{Davies2}
E.~B. Davies.
\newblock Sharp boundary estimates for elliptic operators.
\newblock {\em Math. Proc. Cambridge Philos. Soc.}, 129(1):165--178, 2000.

\bibitem{HenrotCiber}
Antoine Henrot.
\newblock Continuity with respect to the domain for the {L}aplacian: a survey.
\newblock {\em Control Cybernet.}, 23(3):427--443, 1994.
\newblock Shape design and optimization.

\bibitem{HenrotVert}
Antoine Henrot.
\newblock {\em Extremum problems for eigenvalues of elliptic operators}.
\newblock Frontiers in Mathematics. Birkh\"auser Verlag, Basel, 2006.

\bibitem{lemenant1}
Antoine Lemenant and Emmanouil Milakis.
\newblock Quantitative stability for the first {D}irichlet eigenvalue in
  {R}eifenberg flat domains in {$\Bbb R^N$}.
\newblock {\em J. Math. Anal. Appl.}, 364(2):522--533, 2010.

\bibitem{Pang}
M.~M.~H. Pang.
\newblock Approximation of ground state eigenvalues and eigenfunctions of
  {D}irichlet {L}aplacians.
\newblock {\em Bull. London Math. Soc.}, 29(6):720--730, 1997.

\end{thebibliography}

\end{document}